\documentclass[10pt]{amsart}

\usepackage{amsmath,amssymb,dsfont,mathrsfs}
\usepackage{amsthm}
\usepackage{url}

\usepackage{pgf,tikz,pgfplots}
\pgfplotsset{compat=1.13}
\usetikzlibrary{arrows}
\usepackage{multicol}
\usepackage{bm}
\usepackage[all,cmtip]{xy}
\usepackage{tcolorbox}
\usepackage{tikz-cd}

\usepackage{mathtools}
\usepackage{colonequals}

\usepackage{enumitem}
\usepackage{latexsym}

\usepackage[pagebackref=true]{hyperref}
\hypersetup{colorlinks=true,urlcolor=magenta,citecolor=magenta,linkcolor=magenta}

\renewcommand*{\backref}[1]{}
\renewcommand*{\backrefalt}[4]{%
  \ifcase #1 %
    \relax
  \or
    $\uparrow$#2.%
  \else
    $\uparrow$#2.%
  \fi%
}

\usepackage{makecell}
\newcommand{\mathbbords}{\mathbb}

\newcommand{\F}{\mathbbords{F}}

\newcommand{\Fpbar}{\overline{\mathbb{F}}_p}

\newcommand*{\MAGMA}{\textsf{Magma}}
\newcommand{\defi}[1]{\emph{#1}}

\usepackage{color}

\usepackage{float}
\usepackage{verbatim}

\theoremstyle{plain}

\counterwithin{equation}{section}
\counterwithin{figure}{section}
\counterwithin{table}{section}

\theoremstyle{plain}

\newtheorem{theorem}{Theorem}[section]
\newtheorem{proposition}[theorem]{Proposition}
\newtheorem{prop}[theorem]{Proposition}

\newtheorem{lemma}[theorem]{Lemma}
\newtheorem{corollary}[theorem]{Corollary}

\theoremstyle{definition}
\newtheorem{definition}[theorem]{Definition}
\newtheorem{example}[theorem]{Example}

\theoremstyle{remark}
\newtheorem{remark}[theorem]{Remark}

\title{Computing Invariants of Artin-Schreier Curves}

\author[J. Duque-Rosero]{Juanita Duque-Rosero}
\address{Juanita Duque-Rosero, Department of Mathematics and Statistics, Boston University, 665 Commonwealth Ave, Boston, MA 02215, USA}
\curraddr{}
\email{juanita@bu.edu}
\urladdr{\url{https://juanitaduquer.github.io}}

\author[E. Lorenzo Garc\'ia]{Elisa Lorenzo Garc\'ia}
\address{Elisa Lorenzo Garc\'ia, Aix-Marseille Univ, CNRS, I2M, Marseille, France}
\curraddr{}
\email{elisa.lorenzo-garcia@univ-amu.fr }
\urladdr{\url{https://sites.google.com/site/elisalorenzo}}

\author[B. Malmskog]{Beth Malmskog}
\address{Beth Malmskog, Department of Mathematics and Computer Science, Colorado College, 14 E Cache la Poudre, Colorado Springs, CO 80903, USA}
\curraddr{}
\email{bmalmskog@coloradocollege.edu}
\urladdr{\url{https://malmskog.wordpress.com}}

\author[R. Scheidler]{Renate Scheidler}
\address{Renate Scheidler, Department of Mathematics and Statistics, University of Calgary, 2500 University Drive NW, Calgary, Alberta, Canada T2N 1N4}
\curraddr{}
\email{rscheidl@ucalgary.ca}
\urladdr{\url{https://cpsc.ucalgary.ca/~rscheidl}}

\begin{document}


\begin{abstract}
We present an algorithmic framework for computing generators for the ring of invariants of an Artin-Schreier curve.
We give explicit invariants for almost all Artin-Schreier curves of genus up to~8 in standard form, and for a handful of curves of higher genus.
\end{abstract}

\maketitle

\section{Introduction} \label{sec:Intro}

Artin-Schreier extensions \cite{artin1927kennzeichnung} are degree~$p$ Galois extensions of a field of prime characteristic $p$. For almost a century, their associated curves have undergone intense research due to their large point counts and often unusual automorphism groups \cite{van1992reed}, \cite{BHMSSV}, interesting arithmetic statistical behavior \cite{entin2012distribution,bucur2012distribution, bucur2016statistics} and compelling applications to coding theory \cite{van1991artin}. 

Invariant theory, an often challenging area with a venerable history dating back to the 19$^{\textbf{th}}$ century,  investigates polynomial functions that remain unchanged under the action of a group. Hilbert proved that the ring of invariants of a linearly reductive group acting on a polynomial ring over a characteristic zero field is finitely generated. Under mild conditions, the evaluation of these polynomials characterizes the orbits produced by a group action. Initially largely done via hand calculation, modern invariant theory utilizes group representations, coordinate rings of categorical quotients, moduli spaces, and other advanced tools in algebraic geometry.

In \cite{win6}, the authors computed systems of invariants over $\Fpbar$ for Artin-Schreier curves of genus $g\in\{3,4\}$ for $p>2$.  In this article, we present general algorithms to compute invariants of Artin-Schreier curves and explicitly parameterize the moduli spaces of Artin-Schreier curves in characteristic $p>2$ for $g\in\{5,6,7,8\}$ and two higher genera.  Our computational results are summarized in Table~\ref{tab:smallgenustable}. The \MAGMA~\cite{magma} implementation for these computations is available at~\cite{githubRepo}.

\begin{table}[ht]
\setlength{\tabcolsep}{10pt}
\begin{tabular}{l|c|c|c|l|c|c}
$g$ & $p$ & $D$ & $s$ & $(\vec{d}, \ \dim \mathcal{AS}_{g,\vec{d}})$ &$\#$ invariants & \shortstack{running time \\ (seconds)} \\ \hline
3   & 3   & 3   & 0   & (\{4\}, 1)        & 1  & 0.110             \\
    &     &     & 2   & (\{2,1\}, 2)      & 3  & 0.000             \\
    & 7   & 1   & 0   & (\{2\},0)         & 0  & 0.010             \\ \hline
4   & 3   & 4   & 0   & (\{5\}, 2)        & 3  & 0.000             \\
    &     &     & 2   & (\{2,2\}, 3)      & 3  & 0.000             \\
    &     &     & 4   & (\{1,1,1\}, 3)    & 4  & 0.030             \\
    & 5   & 2   & 0   & (\{3\}, 1)        & 1  & 0.000             \\
    &     &     & 4   & (\{1,1\}, 1)      & 1  & 0.000             \\ \hline
5   & 3   & 5   & 2   & (\{4,1\}, 3)      & 8 & 0.010             \\
    &     &     & 4   & (\{2,1,1\}, 4)    & 6  & 0.020             \\
    & 11  & 1   & 0   & (\{2\}, 0)        & 0  & 0.000             \\ \hline
6   & 3   & 6   & 0   & (\{7\}, 3),          & -&see Example~\ref{E:r0d7p3}         \\
    &     &     & 2   & (\{5,1\}, 4),     & 20 & 0.050             \\
    &     &     &     & (\{4,2\}, 4)      & 14 & 0.010             \\
    &     &     & 4   & (\{2,2,1\}, 5)    & 9  & 0.010             \\
    &     &     & 6   & (\{1,1,1,1\}, 5) & 6 &see Section~\ref{sec:fourPoles}              \\
    & 5   & 3   & 0   & (\{4\}, 2)        & 5  & 0.000             \\
    &     &     & 4   & (\{2,1\}, 2)      & 5  & 0.000             \\
    & 7   & 2   & 0   & (\{3\}, 1)        & 1  & 0.000             \\
    &     &     & 6   & (\{1,1\}, 1)      & 1  & 0.000             \\ 
    & 13  & 1   & 0   & (\{2\}, 0)        & 0  & 0.000             \\ \hline
7   & 3   & 7   & 0   & (\{8\}, 4)        & 24 & 2.540             \\
    &     &     & 2   & (\{5,2\}, 5)      &35 & 0.120             \\
    &     &     & 4   & (\{4,1,1\}, 5)    & 9  & 0.010             \\
    &     &     &     & (\{2,2,2\}, 6)    & - & see~\S\ref{subsec:3sameOrder}            \\
    &     &     & 6   & (\{2,1,1,1\}, 6)   &5&see Example~\ref{ex:2,2,2,3}            \\  \hline
8   & 3   & 8   & 2   & (\{7,1\}, 5)      & 40 & 19.710            \\
    &     &     &     & (\{4,4\}, 5)      & 32    & 1.840          \\
    &     &     & 4   & (\{5,1,1\}, 6)    & 12 & 0.010             \\
    &     &     &     & (\{4,2,1\}, 6)    & 21 & 0.010             \\
    &     &     & 6   & (\{2,2,1,1\}, 7)  & - & see Example~\ref{eg:2,2,1,1}                   \\
    &     &     & 8   & (\{1,1,1,1,1\} ,7)  & - &        see~Section~\ref{sec:5poles}       \\
    & 5   & 4   & 4   & (\{3,1\}, 3)   
    & 16 & 0.000             \\
    &     &     &     & (\{2,2\}, 3)      & 13 & 617.550             \\
    &     &     & 8   & (\{1,1,1\}, 3)    & 6  & 0.020             \\
    &17   & 1   &0    &(\{2\}, 0)         & 0  & 0.000
\end{tabular} \medskip
\caption{Partitions $\vec{d}$
as in 
\eqref{eq:ASgenus} for irreducible components 
of $\mathcal{AS}_{g,s}$ and their dimensions for $3\leq g\leq 8$ and $p\geq 3$.  Also, number of generators for the invariant ring 
and running time. }

\label{tab:smallgenustable}
\end{table}

\section{Background and Preliminaries}\label{sec:Background}

In this section we summarize the required notions and results on Artin-Schreier curves and invariant theory.

\subsection{Artin-Schreier curves}\label{sec:AScurves}

Throughout, let $\F_p$ be a finite field of odd prime order $p$ and $\Fpbar$ a fixed algebraic closure of $\F_p$. An \defi{Artin-Schreier curve} is a curve over $\Fpbar$ with an affine model of the form 
\begin{equation}\label{eqn:ASgeneralform}
C_f : y^p-y=f(x),
\end{equation}
where $f(x)\in\Fpbar(x)$ and $f(x) \ne z^p - z$ for any $z \in \Fpbar(x)$. We exclude $p = 2$ from consideration as Artin-Schreier curves in characteristic 2 are hyperelliptic; their invariants and moduli spaces are  well understood.

Suppose that $f(x)$ has $r+1$ distinct poles (with $r \ge 0$) of respective orders $d_1, d_2, \ldots d_{r+1}$.  We refer to the number of poles of a given order $d_i$ as the pole multiplicity. By \cite[Lemma 3.7.7~(b)]{stichtenoth2009algebraic}, we may assume that no $d_i$ is a multiple of $p$. 

By \cite[Proposition 3.7.8~(d)]{stichtenoth2009algebraic}, the genus $g = g(C_f)$ of $C_f$ is 
\begin{equation}\label{eq:ASgenus}
    g=\frac{p-1}{2}D, \quad \mbox{where} \quad D = 
     r - 1 + \sum_{i=1}^{r+1} d_i. 
\end{equation}
By \cite{subrao1975p}, the \emph{$p$-rank} of $C_f$, i.e.\ the integer $s$ such that the $p$-torsion of the Jacobian of~$C_f$ has order $p^s$, is $s=r(p-1)$ and satisfies $0 \le s \le g$. Denote by $\mathcal{AS}_{g}$ the moduli space of Artin-Schreier $\Fpbar$-curves of genus $g$ and by $\mathcal{AS}_{g,s}$ the locus corresponding to Artin-Schreier $\Fpbar$-curves of genus $g$ with $p$-rank exactly $s$. The space $\mathcal{AS}_{g}$ is stratified according to $p$-rank as follows.

\begin{theorem}[{\cite[Theorem~1.1]{PriesZhu2012}}]\label{theorem:PriesZhuDimension}
Let $g=D(p-1)/2$ with $D\geq 1$ and $s=r(p-1)$ with $r\geq 0$.     
    \begin{enumerate}
        \item The set of irreducible components of $\mathcal{AS}_{g,s}$ is in bijection with the set of partitions $\vec{d} = \{d_1, d_2, \ldots, d_{r+1}\}$ of $D+1-r$ 
        such that $p \nmid d_i$ for $1 \le i \le r+1$.
        \item For any partition $\vec{d} = \{d_1, d_2, \ldots, d_{r+1}\}$
        as given in part (1), the corresponding irreducible component $\mathcal{AS}_{g,\vec{d}}$ of $\mathcal{AS}_{g,s}$ has dimension 
        \[ \dim \mathcal{AS}_{g,\vec{d}}=D-1-\sum_{j=1}^{r+1} \lfloor d_j/p \rfloor. \]
    \end{enumerate}
\end{theorem}

By \cite[Lemma~2.15]{win6}, every isomorphism between Artin-Schreier curves $C_f$ and $\widetilde{C}_{\tilde{f}}$ as given in~\eqref{eqn:ASgeneralform} whose corresponding function field isomorphism $\Fpbar(\tilde{C}_{\tilde{f}})\to\Fpbar(C_f)$ fixes $\Fpbar(x)$ has the form  
\begin{equation}\label{eq:ASisom}
\phi(x, y) = \left ( M(x), \lambda y + h(x) \right ) , \quad M(x) = \frac{\alpha x + \beta}{\gamma x + \delta},
\end{equation} 
where $\alpha, \beta, \gamma, \delta \in \Fpbar$ with $\alpha\delta-\beta\gamma \in \Fpbar^{\times}$, $\lambda \in \F_p^{\times}$, and $h(x) \in \Fpbar(x)$. 

To facilitate invariant computation, we only consider Artin-Schreier curves in \defi{standard form} as described in \cite[Theorem 3.3]{win6} and refined herein this paper. Every Artin-Schreier curve $C_f$ as given in~\eqref{eqn:ASgeneralform} is isomorphic to an Artin-Schreier curve in standard form, where the isomorphism sends the three poles of largest orders to the poles $P_\infty$, $P_0$, and $P_1$ of the respective functions~$x$, $1/x$, and $1/(x-1)$ in $\Fpbar(x)$. It also applies certain normalizations and eliminates appropriate $p$-power monomials.

\begin{theorem}[{\cite[Theorem~3.3]{win6}}]\label{T:normal}
Let $p$ be an odd prime and $C_f$ an Artin-Schreier curve over $\Fpbar$ as given in~\eqref{eqn:ASgeneralform} such that $f(x)$ has $r+1$ poles of respective orders $d_1 \ge d_2 \ge \ldots \ge d_{r+1}$. Then $C_f$ is isomorphic to an Artin-Schreier curve 
$C_g: y^p - y = g(x)$
in \defi{standard form} with $g(x) \in \Fpbar(x)$ given as follows.

\begin{enumerate}
\item \label{case:r=0}\textbf{Case $r = 0$:} $g(x) = F(x) \in \Fpbar[x]$ is monic of degree $d_1$, a multiple of $x^2$, and no monomial appearing in $F(x)$ has an exponent that is divisible by $p$. 

\item \label{case:r=1}\textbf{Case $r = 1$}: $\displaystyle g(x) = F(x) + G \left ( \frac{1}{x} \right )$,
where $F(x), G(x) \in \Fpbar[x]$, $F(x)$ is monic, $\deg(F) = d_1$, $\deg(G) = d_2$, and no monomial appearing in $F(x)$ or $G(x)$ has an exponent 
that is divisible by $p$.

\item \label{case:r>=2}\textbf{Case $r \ge 2$}: $\displaystyle g(x) = F(x) + G \left (\frac{1}{x} \right ) + H \left ( \frac{1}{x-1} \right ) + J(x)$, where $F(x)$, $G(x)$, $H(x) \in \Fpbar[x]$, $\deg(F) = d_1$, $\deg(G) = d_2$, $\deg(H) = d_3$, either $J(x) = 0$ or 
\[ J(x) = \sum_{i=4}^{r+1} \frac{J_i(x-\theta_i)}{(x-\theta_i)^{d_i}}, \] 
with $\theta_i \in \Fpbar \setminus \{ 0, 1 \}$, $J_i(x) \in \Fpbar[x]$ non-zero, $\deg(J_i) < d_i$, and no monomial appearing in $F(x)$, $G(x)$, $H(x)$, or any of the polynomials $x^{d_i} J_i(x^{-1})$ has an exponent that is divisible by~$p$, for $4 \le i \le r+1$.
\end{enumerate}
\end{theorem}

For consistency, we will use the following notation for the polynomials appearing in Theorem~\ref{T:normal} throughout.
\begin{equation} \label{eq:notation}
    F(x) = \sum_{i=1}^{d_1} a_i x^i, \quad G(x) = \sum_{i=1}^{d_2} b_i x^i, \quad H(x) = \sum_{i=1}^{d_3} c_i x^i, 
\end{equation}
with $a_i, b_i, c_i \in \Fpbar$, $a_{d_1}, b_{d_2},c_{d_3} \ne 0$, and $a_i = b_i = c_i = 0$ when $p \mid i$. If any of the polynomials is monic and contains only one term besides the leading term, we may omit its subscript; for example, we may write $F(x) = x^3 + a_2x^2$ as $F(x) = x^3 + ax^2$. 

The standard form in Theorem~\ref{T:normal} applies to any Artin-Schreier curve. However, in many cases, we will use alternative standard forms that lend themselves far better to computations.

\begin{remark} \label{rem:auto}
    If $\phi$ as given in \eqref{eq:ASisom} is an isomorphism between curves in standard form, then by \cite[Lemma 3.4]{win6}, the M\"obius transformation $M(x)$ uniquely determines $h(x)$ up to a multiple in $\F_p$. This will also hold for all the alternative standard forms introduced throughput the paper. 
    In this case, we will write $\phi(x,y) = \phi_{M,\lambda}(x,y)$. Moreover, we will disregard (post- or pre-)composition of $\phi$ with any power of standard form preserving automorphisms.
\end{remark}

\subsection{A brief introduction to invariant theory} \label{sec:Invariants} 
 
We follow the description in \cite{win6} and the classical references
\cite{basson, DK02, DK, Dol03, Eis95, Eis05}. 

Let $K$ be an algebraically closed field and $G$ be a linear algebraic group defined over $K$, acting on an algebraic variety $X$ over $K$. This action defines an action on $K[X]$ by $(g\cdot f)(x)=f(g^{-1}\cdot x)$ for all $x\in X$, $f\in K[X]$, and $g\in G$. An element $f\in K[X]$ is an \defi{invariant} for $X$ if $g\cdot f=f$ for all $g\in G$. The algebra of invariants is $K[X]^G\colonequals\{f\in K[X]:\,g\cdot f=f,\forall g\in G\}$. If $G$ is a geometrically reductive group acting on a finitely generated algebra $R$ over $K$, then $R^G$ is finitely generated as a $K$-algebra (Noether-Artin, Hilbert, Haboush–Nagata). 
In particular this result applies to finite groups in any characteristic.

\begin{lemma}[Noether, Fleischmann, Benson, Fogarty, {\cite[Cor. 3.8.4]{DK02}}]
\label{lem:bound} With the previous notation, if $\operatorname{char}(K)$ does not divide $|G|$, then there exists a set of generators of $K[V]^G$, all of degree smaller or equal than $|G|$.
\end{lemma}

\begin{lemma}[Cor. 0.2 in \cite{symonds}]\label{lem:boundp} Let $G$ be a non-trivial finite group acting linearly on $R=k[ x_1,\dots,x_n]$ with $n>1$. Then $R^G$ is generated by invariants of degree  $\leq n(|G|-1)$  and the relations between the generators are generated in degrees at
most $2n(|G| - 1)$.  
\end{lemma}

These bounds are key to designing algorithms for computing generators of invariant rings, such as \cite[Algorithms 3.5.4 and 3.7.2]{DK02} for the non-modular case. They can be extended to work in the modular case as discussed in \cite[\S 3.3 and 3.4.2]{DK02}; specifically in Algorithm 3.7.5 in loc.\ cit.\ and in \cite{DK}. These algorithms are implemented in \MAGMA~\cite{magma}. 

Denote by $\mathfrak{M}_m^+$ the set of primitive (not a power of another monomial) monomials of positive degree in the $m$ variables. Given $\mu\in \mathfrak{M}_m^+$, we denote by $s_{k,\mu}$ the symmetric function $s_k(\mu(x_{1,1},\ldots x_{m,1}),\dots, \mu(x_{1,n},\ldots, x_{m,n}))$. Let $\mathfrak{S}_n$ be the symmetric group on $n$ items. A describe a specific generating set is given as follows.  

\begin{theorem}[{\cite[Thm. 1]{multisym}}]\label{thm:multisym} Let $K$ be any field and set $R=K[x_{i,j}:i=1\dots m,j=1\dots n]$ and consider $G=\mathfrak{S}_n$ acting as $\sigma(x_{i,j})=x_{i,\sigma(j)}$. Assume $m > 1$. Then $R^G$ is generated by $s_{k,\mu}$ with $k=1\dots m$ and $\mu\in \mathfrak{M}_m^+$ with $k\cdot\deg(\mu)\leq m(n-1)$.
\end{theorem}

\begin{corollary}\label{multisymOrbits} Let $G = \{ g_1 \ldots, g_n\}$ be a finite group acting linearly on $R=K[x_1,...x_n]$, and define $x_{ij}=g_j(x_i)$.  Then $R^G$ is generated by multisymmetric polynomials on the $x_{ij}$; in particular, by the generators in Theorem~\ref{thm:multisym}. 
\end{corollary}

\begin{remark}\label{GnotGroup} The actual action of $G$ in Corollary~\ref{multisymOrbits} is not relevant as we can directly compute invariants from the $G$-orbits. In fact, this technique can even be applied to ``orbits'' without specifying any group action. Given $m$ tuples $[x_{1i},...,x_{ni}]$ ($1 \le i \le m$), with $x_{ij} \in R$ for a given ring $R$, the multisymmetric polynomials on the $x_{ij}$ are polynomials $f\in R[x_1,...,x_n]$ such that $f(x_{1i},...,x_{ni})=f(x_{1j},...,x_{nj})$ for all $i,j\in \{1,...,m\}$. These multisymmetric polynomials are generated by the polynomials described in Theorem~\ref{thm:multisym}.
\end{remark}

In our setting, we consider the case where the variety $X$ is the set of coefficients parameterizing Artin-Schreier curves $C:f :
y^p-y=f(x)$, $f(x) \in \Fpbar[x]$, with a fixed number of poles and multiplicities. 
Here, the acting group is $G=\operatorname{PGL}_2(\Fpbar)\times\mathbb{F}_p^\times$ as explained in Remark~\ref{rem:auto}. This is an infinite group, so the computation of invariants is especially difficult. A standard technique is to determine the associated Hilbert series (see, for example, \cite[Def.\ 2.5]{win6}), then compute sufficiently many invariants, and  check that they generate the entire space by comparing dimensions. While this is \emph{a priori} doable, it is unclear how to systematically produce invariants and how to convert this into a general algorithm that always produces generators. Our strategy therefore is to reduce the problem to having to consider only finite groups. We follow the treatment of \cite[Section 3.2]{LRS20}:

\begin{definition} Let $G$ be a linear algebraic group acting on an irreducible quasiprojective variety $X$. Let $f : H\rightarrow G$ be a group morphism, through which we
consider $H$ to act on $X$. Let $Z$ be a closed subvariety of $X$. Then $Z$ is called a $(G, H)$-section of the action of $G$ on $X$ if the following hold.
\begin{itemize}
\item[(i)] The stabilizer of $Z$ in $X$ is the image of $f$;
\item[(ii)] There exists an open subset $Z_1$ of $Z$ such that any two points of $Z_1$ which
are $G$-equivalent in $X$ are in fact $H$-equivalent in $Z$;
\item[(iii)] $X$ is the closure of $G\cdot Z$.
\end{itemize}
\end{definition}

\begin{proposition}[\mbox{\cite[Section 3]{GV78}}] Suppose that $Z$ is a $(G, H)$-section of $X$. Then the
canonical restriction arrow $K(Y)^G \rightarrow K(Z)^H$ between fields of rational functions is
an isomorphism. If we additionally assume that
\begin{itemize}
\item[(i)] $X$ is an affine normal variety;
\item[(ii)] $G$ is a linear algebraic group that does not admit any non-trivial character;
\item[(iii)] any closed orbit of $G$ in $X$ intersects $Z$,
\end{itemize}
then the canonical restriction arrow $K[X]^G \rightarrow K[Z]^H$ between rings of regular functions is an isomorphism.
\end{proposition}

In many relevant cases, we will be able to work with finite subgroups $H=G_0\times\mathbb{F}_p^{*}$ of $G=\operatorname{PGL}_2(\Fpbar)\times\mathbb{F}_p^\times$, and closed subvarieties $Z$ corresponding to the parameters of Artin-Schreier curves in a particular standard form. 
First, we compute the action of $G_0$ and subsequently the action of $\mathbb{F}_p^\times$. 

\begin{definition} (\cite[\S 2.3.1]{DK02}) Let $G$ be a reductive group and $X$ a variety on which $G$ acts regularly. The \defi{categorical quotient} $X//G$ is the affine variety whose coordinate ring is $K[X]^G$. 
The surjective morphism $\pi:\,X\rightarrow X//G$ corresponding to the inclusion of coordinate rings, which
The morphism $\pi$ 
is a geometric quotient if there is a one-to-one correspondence between $G$-orbits and points of $X//G$. 
\end{definition}

\begin{theorem} (\cite[§0.2]{MumfordFogartyKirwanGIT}) Let $
G$ be a finite group acting on a variety 
$X$ over an algebraically closed field 
$k$. The morphism $\pi:\,X\rightarrow X//G=\operatorname{Spec}(k[X]^G)$
is a geometric quotient if and only if the action of $G$ on $X$ is free, i.e. all stabilizers $G_x$
 are trivial.
\end{theorem}

For groups $G$ and $H$, let $X$ be a $G$-variety and $Y$ an $H$-variety such that $X//G$ and $Y//H$ are geometric quotients. If $X//G\simeq Y//H$, then $K[X]^G\simeq K[Y]^H$.  

\begin{proposition}\label{prop:subgroup} Let $H<G$, then $K[X]^G=(K[X]^H)^{G}$. If $H\lhd G$, then $K[X]^G=(K[X]^H)^{G/H}$. 
\end{proposition}

The action corresponding to $1\times\mathbb{F}_p^*$ (called the $\lambda$-action) is easy to compute because it arises from a finite abelian (cyclic) group of cardinality coprime to $p$. So we can for instance use \cite[Algorithm 2.7.3]{sturmfelsIT} to compute this action. We will sometimes avoid its explicit computation since it is straightforward and only complicates things overall.
 
In some  cases, and in order to produce general results, we will compute invariants \emph{ad hoc} instead of using \MAGMA. In these instances, we will need to prove that the invariants we computed generate the full invariant ring. 

\begin{definition} A subset $S\subseteq K[X]^G$ is said to be \defi{separating} if it satisfies the following property. For any two points $x,y\in X$, if there exists an invariant $f\in K[X]^G$ with $f(x)\neq f(y)$, then there exists an element $g\in S$ with $g(x)\neq g(y)$.
\end{definition}

\begin{definition}
Let $A\subseteq K[X]$ be a subalgebra of a polynomial ring of positive characteristic $p$.
Then the algebra
$
\hat{A}=\{ f\in K[X]:\,f^{p^r}\in A \text{ for some } r\in \mathbb{N}\}\subseteq K[X]
$
is the \defi{purely inseparable closure} of A in $K[X]$.
\end{definition}

\begin{theorem}[{\cite[Theorem.~2.3.15]{DK02}}] Let $X$ be an affine variety and $G \subseteq \operatorname{Aut}(K[X])$ a subgroup of the automorphisms of the coordinate ring $K[X]$. Then there exists a finite
separating set $S\subseteq K[X]^G$.
\end{theorem}

\begin{theorem}[{\cite[Theorem. 2.3.12]{DK02}}]\label{thm:reconstruct}  Let $G$ be a finite group and let $V$ be an $\Fpbar$-rational representation of~$G$.
Let $A\subseteq \Fpbar[V]^G$ be a finitely generated, graded, separating subalgebra. Then $\Fpbar[V]^G=\hat{\tilde{A}}$, the purely inseparable closure of the normalization of $A$.
\end{theorem}

A set of invariants that allow reconstruction of a point as described in Corollary~\ref{thm:reconstruct} is referred to as a \defi{reconstructing system}.

\begin{remark} \label{genericallyreconstruction} Let $
G$ be a finite group acting on a variety 
$X$ over an algebraically closed field 
$K$. A recent result \cite[Proposition 1]{ReimersSezer2025} states that generically reconstructing systems generate $K(X)^G$.  
\end{remark}

\section{Main Algorithm and Results}\label{sec:MainAlg}

\subsection{General Method}\label{subsec:methods}
The group of isomorphisms as given in \eqref{eq:ASisom} is infinite, but by considering isomorphisms of standard models we are able to rigidify the problem of computing $\Fpbar$-invariants for Artin-Schreier curves and reduce to a finite subgroup $G$, a much more tractable situation.  To describe an Artin-Schreier curve~$C_f$, we use the coefficients of its standard form as described in Theorem~\ref{T:normal}, or variants thereof. Then the invariant ring is a subring of the polynomial ring $R$ on variables generated by the coefficients of the standard form of $C_f$.  

When $G$ acts linearly on the coefficients of a standard form $C_g: y^p-y = g(x)$, we can directly apply results about linear actions from Section~\ref{sec:Invariants}, including Lemma~\ref{lem:bound}, and use built-in methods  from \MAGMA~in our code to find invariants. This occurs, for example, when $g(x)$ has exactly three poles or at least three poles of multiplicity one (see Section~\ref{sec:3poles}).  In that case, we compute a matrix representation of the action of isomorphisms that preserve the standard form in terms of the coefficients of $g(x)$.  Then we compute invariants of the group generated by these matrices, using \MAGMA's invariant theory package for finite groups.

Alternatively, we can also compute invariant rings by generating them explicitly.  This idea is especially useful when the action of the isomorphisms that preserve the standard form is not linear.  We let $R$ be the polynomial ring over $\Fpbar$ on the variables given by the coefficients and the poles of $g(x)$. Then each isomorphism of the curve $C_g$ that returns a standard form gives the new coefficients as rational functions on the variables of $R$. In some cases, the resulting rational functions are polynomials.  In other cases, they are true rational functions.  For example, if $C_g$ has more than three poles but not at least three poles of multiplicity one, the group $G$ of isomorphisms of standard forms will allow some permutation of the distinguished poles  $P_{\infty}$,$P_0$ and $P_1$, and the result of such a M\"{o}bius transformations will be a model with coefficients that are rational functions involving the coefficients of the additional poles (for example $\theta$ if $g(x)$ has a pole at $x=\theta$).

Unfortunately, the isomorphisms preserving a standard model do not always form a group (e.g. see Example~\ref{E:r0d7p3} and Section~\ref{sec:5poles}). In these cases, we cannot properly compute invariants. However, we can still compute ``specializations'' of invariants using Remark~\ref{GnotGroup}. 

\begin{theorem}\label{thm:GnotGroup} Let $G$ be a group acting linearly on $R=K[y_1,...,y_r]$. Let $\phi:\,R_0=K_0[x_1,...,x_n]\hookrightarrow R$ be a ring homomorphism such that the orbit of each element of $R$ by $G$ intersects $\phi(R_0)$ generically in a finite number $m$ of elements. Define the finite set $G_f=\{\sigma\in G:\,^\sigma f\in R_0\}$ and $\operatorname{Orb}(f)=\{\,^\sigma f:\,^\sigma\in G_f\}$. Then for any $I\in R^G$, $I(\phi(f))$ is a multisymmetric polynomial on the $x_{ij}$ as in Remark~\ref{GnotGroup}. A representative of $\phi(f)$ by the action of $G$ can be reconstructed from the multisymmetric functions of $f$, in particular by those in Theorem~\ref{thm:multisym}; and a representative of $f$ can be reconstructed by the values of $I(\phi(f))$ for all $I$ in a generator system of $R^G$. In addition, invariants in $R^G$ can be constructed by pushing-forward these multisymmetric functions as in Remark 4.4 in \cite{win6}.
\end{theorem}

\subsection{Road map}
We provide a road map for computing invariants of Artin-Schreier curves. Depending on the parameters, we re-order the pole orders appropriately and choose three  \emph{distinguished}poles of respective orders $d_1$, $d_2$, $d_3$ that are mapped  $P_{\infty}$, $P_0$, and $P_1$ via an isomorphism that produces a suitable standard form. In each case, we refer to the section where the corresponding invariant ring is discussed. 

\bigskip

\noindent {\bf Input:} A prime $p$ and a list of pole orders $\{d_1,d_2,\dots,d_{r+1}\}$.  

\noindent {\bf Output:} A set of reconstructing invariants for the corresponding space of Artin-Schreier curves defined by the prime $p$ and the list of pole orders $\{d_1,d_2,\dots,d_{r+1}\}$ suitably rearranged.

\begin{enumerate}
    \item If there is exactly one pole, proceed via Section~\ref{sec:onepole}.
    \item If there are exactly two poles, sort so $d_1\geq d_2$, then  proceed via Section~\ref{sec:twopoles}.
    \item If there are exactly three poles,
    proceed via Section~\ref{sec:3poles}.  
    \item Else if there are at least three pole orders of multiplicity one (outside case (3)), let $d_1 \ge d_2\ge d_3$ be the largest orders of the poles of multiplicity one and $d_4 \geq \cdots \ge d_{r+1}$ are the remaining pole orders. Then proceed via Theorem~\ref{thm:3unique} in general or Corollary~\ref{cor:allunique} if all poles have multiplicity one. 
    \item Else if there are  two pole orders of multiplicity one, Let $d_1 \ge d_2$ be these pole orders and $d_3 \geq \cdots \ge d_{r+1}$ are the remaining pole orders. Then proceed via Proposition~\ref{prop:2pole} and Theorem~\ref{thm:11+}.
    \item Else if there is a pole order of multiplicity three, sort so $d_1 = d_2 = d_3$ is the largest pole order of multiplicity three and $d_4 \ge \cdot s\ge d_{r+1}$ are the remaining pole orders. Then proceed vis See Section~\ref{ssec:3plus}.
    \item Else choose $d_1$ to be the largest pole order among the poles of smallest multiplicity and sort so that $d_2 \ge \cdots \ge d_{r+1}$. Then proceed via Sections~\ref{sec:fourPoles} and~\ref{sec:5poles} for general methods.
\end{enumerate}

\subsection{Results}\label{subsec:results}
We implemented and ran our algorithm in \MAGMA~for all families of Artin-Schreier curves of two and three poles and for one pole of order $d\not\equiv 1 \pmod p$. We also have specific examples that include all the entries in Table~\ref{tab:smallgenustable}.  The algorithm was run in \MAGMA~V2.28-3 on an Apple M2 with 8 GB memory.  The implementation is available at \cite{githubRepo}.

\section{One Pole}\label{sec:onepole}

For brevity, put $d_1 = d$. Theorem~\ref{T:normal} identifies any standard form as $C_g : y^p-y=g(x)$ where $g(x) \in \Fpbar[x]$ is monic of degree $d_1 = d$ and contains no $p$-power monomials in $x$.  Any standard form preserving isomorphism as given in \eqref{eq:ASisom}  must preserve the polynomial form of $g(x)$ and thus satisfy $\gamma = 0$ and $\delta \in \Fpbar^\times$, so we may take $\delta = 1$. Moreover, comparing leading terms in $x$ of $C_g$ shows that $\lambda = \alpha^d$. Thus, in the one pole setting, all standard form preserving isomorphisms have the form
\begin{equation} \label{eq:end_r=0}
\phi: (x,y) \mapsto (\alpha x + \beta, \alpha^d y + h(x)) , \quad \deg(h) \le \lfloor d/p \rfloor , 
\end{equation}
with $\alpha, \beta \in \Fpbar$, $\alpha^d \in \Fpbar^\times$ and $h(x) \in \Fpbar(x)$. To derive invariants of $C_g$, we distinguish according to whether or not $d-1$ is a multiple of $p$. 

\subsection{One pole of order \texorpdfstring{$\bm{d\not\equiv1\pmod p}$}{d != 1 mod p}}  
\label{sec:r=0dnot1modp}

To substantially simplify computations, we use the alternative standard form proposed by Farnell~\cite[Proposition~2.1.1]{farnell} that removes the $x^{d-1}$-term rather than the $x$-term as was done in part (1) of Theorem~\ref{T:normal}; see loc.\ cit.\ for a proof of its existence and an algorithm to obtain it.

\begin{proposition} \label{prop:normal2}
    Let $p$ be an odd prime and $C_f$ an Artin-Schreier $\Fpbar$-curve as given in~\eqref{eqn:ASgeneralform} such that $f(x)$ has exactly one pole of order $d \not \equiv 1 \pmod{p}$. Then $C_f$ is isomorphic to an Artin-Schreier curve $y^p - y = x^d + F(x)$, where $F(x) \in \Fpbar[x]$, either $F(x) = 0$ or $F(x)$ has degree at most $d-2$, and no monomial appearing in $F(x)$ has an exponent that is divisible by $p$. 
\end{proposition}

Thus, we consider an Artin-Schreier curve 
\begin{equation} \label{eq:C1}  C\colon y^p - y = x^d + F(x) , \qquad F(x) = \sum_{i=1}^{d-2} a_i x^i \in \Fpbar[x], \end{equation}
where $a_i=0$ whenever $p \mid i$.

\begin{prop} \label{prop:r=0easy}
    Suppose $d \not\equiv 1 \pmod{p}$. Then every isomorphism preserving the standard form given in \eqref{eq:C1} is of the form $(x,y) \mapsto (\alpha x, \alpha^d y)$.
\end{prop}
\begin{proof}
Applying \eqref{eq:end_r=0} to \eqref{eq:C1} yields
\begin{equation} \label{eq:C2}
\alpha^d(y^p-y) + h(x)^p - h(x) = (\alpha x+\beta)^d + F(\alpha x + \beta) = \alpha^d x^d + d\alpha^{d-1}\beta x^{d-1} + F_{d-2}(x) ,
\end{equation}
where $F_{d-2}(x) \in \Fpbar[x]$ has degree at most $d-2$. Comparing coefficients of $x^{d-1}$ in \eqref{eq:C2} and \eqref{eq:C1} yields $d\alpha^{d-1}\beta = 0$, forcing $\beta = 0$ and hence $F_{d-2}(x) = F(\alpha x)$. Since $a_i = 0$ for all multiples $i$ of $p$, the right hand side of~\eqref{eq:C2} contains no $p$-power monomials, hence neither does $h(x)^p - h(x)$. Writing $h(x) = \sum_{i=0}^M h_i x^i$ with $M = \deg(h) = \lfloor d/p \rfloor$, we obtain
\[
    h_i^p = 0 \quad \mbox{for $M \ge i > M/p$}, \qquad
    h_i^p - h_{pi} = 0 \quad \mbox{for $M/p \ge i \ge 0$}.
\]
We claim that $h(x) = h_0 \in \F_p$, a constant polynomial in $F_p$. To that end, let $i \in \{ 1, \ldots , M\}$, and let $j_i$ be the smallest integer such that $ip^{j_i} > M/p$. Then 
\[ 0 = h_{ip^{j_i}} = h_{ip^{j_i-1}}^p = h_{ip^{j_i-2}}^{p^2} = \cdots = h_i^{p^{j_i}}, \]
so $h_i = 0$. For the constant coefficient $h_0$ of $h(x)$, we obtain $h_0^p - h_0 = 0$, so $h_0 \in \F_p$. It follows that the isomorphism is the composition of $(x,y) \mapsto (\alpha x, \alpha^d y)$ with a power of the automorphism $(x,y) \mapsto (x, y+1)$. By Remark~\ref{rem:auto}, we may take $h_0 = 0$. 
\end{proof}

When $d \not\equiv 1 \pmod{p}$, the set of generators for the invariant ring of \eqref{eq:C1}, is given as follows. 

\begin{proposition}\label{prop:r=0d!=1modp} With the notation of \eqref{eq:notation}, if $d \not\equiv 1 \pmod{p}$, then the ring of invariants of the curve \eqref{eq:C1} is generated by
$$
\left\{ a_1^{k_1}\cdots a_{d-2}^{k_{d-2}}: \sum k_i\leq d(p-1),\, \sum k_i(d-i)\equiv0\pmod{d(p-1)}\right\}.
$$
\end{proposition}
\begin{proof} The isomorphism of Proposition~\ref{prop:r=0easy} sends each coefficient $a_i$ of $F(x)$ to $\alpha^{i-d}a_i$. In addition, $(\alpha^d)^{p-1}=\lambda^{p-1}=1$ since $\lambda\in \Fpbar^\times$. The bound on the sum of the exponents follows from Lemma~\ref{lem:bound} and the fact that~$p\nmid d$.  The congruence modulo $d(p-1)$ guarantees that the elements are invariant.
\end{proof}

\begin{example}[$r=0$, $p=3$, $d=8$, $g = 7$]  \label{E:r0d8p3}
Here, the standard form is
$$y^3-y=x^8+ a_5x^5 + a_4x^4 + a_2x^2 + a_1x.$$
We obtain that the invariants generating the invariant ring are 
\begin{align*}
    &a_1a_2a_5,\,
    a_2^2a_4,\,
    a_1a_5^3,\,
    a_2a_4a_5^2,\,
    a_4^4,\,
    a_1^4a_4,\,
    a_1^2a_2^3,\,
    a_4a_5^4,\,
    a_1^3a_4^2a_5,\,
    a_1^2a_2a_4^3,\,
    a_1^6a_2,\,
    a_1^2a_4^3a_5^2,\\&
    a_1^6a_5^2,\,
    a_2^8,\,
    a_2^7a_5^2,\,
    a_2^6a_5^4,\,
    a_2^5a_5^6,\,
    a_1^{11}a_5,\,
    a_2^4a_5^8,\,
    a_2^3a_5^{10},\,
    a_2^2a_5^{12},\,
    a_2a_5^{14},\,
    a_1^{16},\,
    a_5^{16}.
\end{align*}
\end{example}

\subsection{One pole of order \texorpdfstring{$d\equiv1\pmod p$}{d = 1 mod p}}  
\label{sec:r=0d=1modp}
For the case $p = 3$, we again have a simpler standard form in many cases.

\begin{proposition} \label{prop:normal3}
    Let $p = 3$ and either $d = 4$ or $d \ge 7$ with $d \equiv 1 \pmod{3}$ and $d \not\equiv 1 \pmod{9}$. Let $C_f$ be an Artin-Schreier $\Fpbar$-curve as given in~\eqref{eqn:ASgeneralform} such that $f(x)$ has exactly one pole of order~$d$. Then $C_f$ is isomorphic to an Artin-Schreier curve 
    \begin{equation} \label{eq:C5} 
    C: y^3 - y = x^d + ax^{d-2} + F(x), 
    \end{equation}
    where $a \in \Fpbar$, $F(x) = 0$ when $d = 4$, $F(x) \in \Fpbar[x]$ has degree at most $d-5$ when $d \ge 7$, and no monomial appearing in $F(x)$ has an exponent that is divisible by $p$.
\end{proposition}
\begin{proof}
    By part (1) of Theorem~\ref{T:normal}, $C_f$ is isomorphic to a curve of the form $C_g: y^3 - y = g(x)$ where $g(x) \in \Fpbar[x]$ is monic of degree $d$. As in the proof of \cite[Lemma 3.1]{win6}, for every $\beta \in \overline{\F}_3$, here exists a polynomial $h_\beta(x) \in \Fpbar[x]$ of degree at most $(d-1)/3$ such that the isomorphism $(x,y) \mapsto (x+\beta, y+h_\beta(x))$ transforms $C_g$ into a curve 
    \[ C_{g_\beta}: y^3 - y = g_\beta(x), \quad g_\beta(x) = g(x+\beta) - h_\beta(x)^3 + h_\beta(x) ,\] 
    where $g_\beta(x)$ contains no terms with monomials $x^{3i}$ for $0 \le i \le (d-1)/3$. In particular, $g_\beta(x)$ contains no terms with monomials $x^{d-1}$ and $x^{d-4}$. It remains to analyze the $x^{d-3}$-term. Writing $g(x) = x^d + \sum_{i=0}^{d-2} a_i x^i$, 
    we see that the coefficient of this term is $a_1 + 2a_2\beta + \beta^3 + h_1$ when $d = 4$ and $a_{d-3} + 2a_{d-2}\beta + \binom{d}{3}\beta^3$ otherwise, where $h_1$ is the linear term of $h_\beta(x)$. Since $\binom{d}{3} \not\equiv 0 \pmod{3}$, we can chose $\beta \in \Fpbar$ such that this quantity vanishes. 
\end{proof}

The result of Proposition~\ref{prop:normal3} may be false when $ d \ge 7$ and $d \equiv 1 \pmod{9}$ because here, $\binom{d}{3}$ is divisible by~$3$. In this case, the coefficient of $x^{d-3}$ in $f(x+\beta)$ is  $a_{d-3} + 2a_{d-2}\beta$. If $a_{d-2} = 0$ and $a_{d-3} \ne 0$, then there is no choice of $\beta$ that eliminates this coefficient.

\begin{corollary} \label{cor:r=0d=1modp}
For $p = 3$ and $d$ satisfying the conditions of Proposition~\ref{prop:normal3}, every isomorphism preserving the standard form given in \eqref{eq:C5} is of the form 
\[ (x,y) \mapsto (\alpha x + \beta, \alpha^d y + h(x)), \quad \deg(h) \le (d-1)/3 ,\]
where $\alpha^{2d} = 1$ and $\displaystyle \beta^2 \in \left \{ 0, \binom{d}{3} a_{d-2} \right \}$.
\end{corollary}
\begin{proof}
    Substituting \eqref{eq:end_r=0} into \eqref{eq:C5}  yields $\alpha^d(y^3-y) + h(x)^3 - h(x)$ equal to
\begin{align*} \label{eq:C6}
 &(\alpha x+\beta)^d + a_{d-2}(\alpha x + \beta)^{d-2} + F(\alpha x + \beta) \\ 
&= \alpha^d x^d + \alpha^{d-1}\beta x^{d-1} + a_{d-2} \alpha^{d-2}x^{d-2} + \alpha^{d-3} \left ( \binom{d}{3} \beta^3 + 2a_{d-2} \beta \right) x^{d-3} \\
& \qquad \qquad +\sum_{i=0}^{d-4} \alpha^i \beta^{d-2-i} \Biggl ( \binom{d}{i} \beta^2 + a_{d-2}\binom{d-2}{i} \Biggr ) x^i+ F(\alpha x + \beta),
\end{align*}
where we use the fact that $d \equiv 1 \pmod{3}$ implies $d-2 \equiv 2 \pmod{3}$ and $\binom{d}{2} \equiv 0 \pmod{3}$. Since $3 \nmid d-3$ and $\deg(h) \le (d-1)/3 < d-3$, neither $h(x)^p$ nor $h(x)$ contain an $x^{d-3}$-term. Comparing coefficients of $x^{d-3}$ yields $\beta \left ( \binom{d}{3} \beta^2 + 2a_{d-2} \right ) = 0$, so $\beta = 0$ or  $\beta^2 =  \binom{d}{3} a_{d-2}$ as $\binom{d}{3}^{-1} \equiv \binom{d}{3} \pmod{3}$.
\end{proof}

\begin{example}[$r=0$, $d=7$, $p=3$, $g = 6$] \label{E:r0d7p3} 
By Proposition~\ref{prop:normal3}, the standard form in this case is 
$y^3 - y = x^7 + a_5 x^5 + a_2 x^2 + a_1 x$.
By Corollary~\ref{cor:r=0d=1modp}, the isomorphisms preserving this form are given by 
$(x,y) \mapsto (\alpha x + \beta, \alpha^7 y + h_2x^2 + h_1x + h_0)$, where $\alpha^{14} = 1$ and $\beta^2 \in \{ 0, 2a_5 \}$. Comparing coefficients yields
\[
    h_2^3  = \alpha^6 \beta, \qquad
    a_5' = \alpha^{-2} a_5, \qquad
    h_1^3 = \alpha^3\beta^2(2\beta^2 + a_5) , \qquad
    a_2' = \alpha^{-5} (a_5\beta^3+a_2) + h_2 , \]
    \[ a_1' = \alpha^{-6}(\beta^6 + 2a_5\beta^4 + 2a_2\beta + a_1) + h_1, \qquad
    h_0^3 - h_0 = \beta(\beta^6 + a_5\beta^4 + a_2\beta + a_1). \]
The equations for $a_5',\,a_2',\,a_1'$ allow to find the orbits of $a_5,\,a_2,$ and $a_1$ under the isomorphisms.  However, these orbits depend on the values $h_i$, which are given by algebraic expressions on $\alpha$ and $\beta$.  This issue made our implementation ran out of memory before producing a set of generating invariants.  In addition, the described set of transformations is not a group in this case. So, the multisymmetric functions computed as in Remark~\ref{GnotGroup} are only specializations of invariants for the infinite group as in Theorem~\ref{thm:GnotGroup}.

\end{example}

For completeness, we sketch the approach for computing generators for the invariant ring in general when $d \equiv 1 \pmod{p}$. Here, we have no choice but to use the standard form of Theorem~\ref{T:normal} because it is unclear which if any terms of degree exceeding 1 can be eliminated; this depends on which binomial coefficients $\binom{d}{i}$ vanish. So consider a curve of the form
\begin{equation} \label{eq:C3} y^p - y = x^d + F(x), \qquad F(x) = \sum_{i=2}^{d-2} a_i x^i \in \Fpbar[x], \end{equation}
where $a_i=0$ whenever $p \mid i$ (so in particular $a_{d-1} = 0$). Applying \eqref{eq:end_r=0} to \eqref{eq:C3}, we obtain $y^p - y = x^d + F'(x)$ where $F'(x)$ satisfies
\begin{equation} \label{eq:C4}
\alpha^d(x^d + F'(x)) = (\alpha x+\beta)^d + F(\alpha x + \beta) - h(x)^p + h(x). 
\end{equation}
We have 
\[ F(\alpha x + \beta) = \sum_{j=0}^{d-2} a_j(\alpha x+\beta)^j 
	   = \sum_{i=0}^{d-2} \left ( \sum_{j=i}^{d-2} a_j \binom{j}{i} \beta^{j-i} \right ) \alpha^i x^i. \]
Write
$ F'(x) = \sum_{i=0}^{d-2} a_i' x^i$ and $h(x) = \sum_{i=0}^{(d-1)/p} h_i x^i$,
where $a_i' = 0$ for $p \mid i$ and for $i = 1$. Comparing coefficients in \eqref{eq:C4} yields $d$ equations of three different types. For indices that are multiples of $p$, these identities determine $h(x)$. For indices $i > 1$ that are not multiples of $p$, they define $F'(x)$. Finally, the equation for $i=1$ produces a polynomial equation for $\beta$ whose coefficients only depend on \eqref{eq:C4}. We discuss each of these cases in turn. 

First, we compare coefficients of $x^{pi}$ for $0 \le i \le (d-1)/p$ in~\eqref{eq:C4}. Since $a_{ip} = 0$ for all $i$, the coefficient of $x^{pi}$ in $F(x)$ is
\[ \alpha^{ip} \sum_{j=ip+1}^{d-2} a_j \binom{j}{ip} \beta^{j-ip} = \alpha^{ip} \beta \sum_{j=1}^{d-ip-2} a_{ip+j} \binom{ip+j}{ip} \beta^j . \]
Now
\[ \binom{ip+j}{ip} = \binom{ip+j}{j} = \frac{(ip+j)(ip+j-1) \cdots (ip+1)}{j!} = 1 \mbox{ in $\F_p$}, \]
so \eqref{eq:C4} yields
\begin{equation} \label{eq:hi}
\alpha^{ip} \Biggl ( \binom{d}{ip} \beta^{d-ip} +  \sum_{j=1}^{d-ip-2} a_{ip+j} \beta^j \Biggr ) 
     = \begin{cases} h_i^p - h_{ip} & \mbox{for $0 \le i \le (d-1)/p^2$}, \\ h_i^p & \mbox{for $(d-1)/p^2 < i \le (d-1)/p$}, \end{cases}
\end{equation}
which can be used to compute certain $p$-powers of  $h_i\alpha^{-i}$ recursively as polynomials in $\beta$. For brevity, put 
\begin{equation} \label{eq:Fq}
\F_q \colonequals \F_p[a_1, a_2, \ldots , a_{d-2}] . 
\end{equation}
\begin{lemma} \label{lem:hi}
For $1 \le i \le (d-1)/p$, put $j_i \colonequals \lfloor \log_p((d-1)/i) \rfloor$. Then $h_i^{p^{j_i}} = \alpha^{-ip^{j_i}}\beta F_i(\beta)$, where $F_i(x) \in \F_q[x]$ has degree at most $(d-ip)p^{j_i-1}-1$. 
\end{lemma}
\begin{proof}
For $(d-1)/p^2 < i \le (d-1)/p$, we have $j_i = 1$. In this case, the assertion follows immediately from the second case of \eqref{eq:hi}.

Now suppose that $1 \le i \le (d-1)/p^2$. Raising the first case in \eqref{eq:hi} to the power $p^{j_i-1}$ shows that $h_i^{p^{j_i}} = h_{ip}^{p^{j_i}-1} - \alpha^{ip^{j_i}} \beta \tilde{F}_i(\beta)$ for some polynomial $\tilde{F}_i(x) \in \F_q[x]$ of degree at most $(d-ip)p^{j_i-1}-1$ in $x$. Since $j_i - 1 = j_{ip}$ for $1 \le i \le (d-1)/p^2$, we obtain inductively that 
\[ h_{ip}^{p^{j_i}-1} = h_{ip}^{p^{j_{ip}}} = \alpha^{(ip)p^{j_{ip}}}\beta F_{ip}(\beta) = \alpha^{ip^{j_i}} \beta F_{ip}(\beta) , \]
where $F_{ip}(x) \in \F_q[x]$ has degree at most 
\[ (d-(ip)p)p^{k_{ip}-1}-1 = (d-ip^2)p^{j_i-2}-1 < (d-ip)p^{j_i-1}-1 . \]
Hence $h_i^{p^{j_i}} = \alpha^{ip^{j_i}} \beta F_i(\beta)$ where $F_i(x) = F_{ip}(x) - \tilde{F}_i(x) \in \F_q[x]$ has degree at most $(d-ip)p^{j_i-1}-1$ as asserted. 
\end{proof}

\section{Two Poles}\label{sec:twopoles}
Assume now that the rational function $f(x)$ in \eqref{eqn:ASgeneralform} has two poles of respective orders $d_1\ge d_2$. In this case, Theorem~\ref{T:normal} produces the standard form 
\begin{equation} \label{eq:r=1}
C \colon y^p-y=F(x)+G \left (1/x \right) 
\end{equation}
where $F(x), G(x) \in \Fpbar[x]$ are as in \eqref{eq:notation}, $F(x)$ is monic of degree $d_1$, $G(x)$ has degree $d_2$, and no monomial in $F(x)$ or $G(x)$ has an exponent a multiple of $p$. 

Every isomorphism~\eqref{eq:ASisom} that preserves \eqref{eq:r=1} must send $P_\infty$ to itself. Analogous to Proposition~\ref{prop:r=0easy}, it is now straightforward to show that all possible isomorphisms that preserve \eqref{eq:r=1} are of the form
$
(x,y)\mapsto(\alpha x,\alpha^{d_1} y),
$
with $\alpha^{d_1}\in\mathbb{F}_p$, i.e. $\alpha^{d_1(p-1)}=1$. Similarly to Proposition~\ref{prop:r=0d!=1modp}, we obtain

\begin{proposition}\label{prop:r=1diff} 
With the notation of \eqref{eq:notation}, if $d_1 \ne d_2$, then the ring of invariants of the curve \eqref{eq:r=1} is generated by
\begin{align*}
\left\{ a_1^{k_1}\dots a_{d_1-1}^{k_{d_1-1}}b_1^{n_1}\dots b_{d_2}^{n_{d_2}} : \right.&\sum k_i + \sum n_i\leq d_1(p-1), \\
&\left.\sum k_i(i-d_1)-\sum n_i(i+d_1)\equiv0 {\pmod{d_1(p-1)}}
\right\}.
\end{align*}
\end{proposition}
\begin{proof} A change of variables $(x,y)\mapsto(\alpha x,\alpha^{d_1} y)$ sends the coefficients $a_i$ of $F(x)$ to $\alpha^{i-d_1}a_i$ and the coefficients $b_i$ of $G(x)$ to $\alpha^{-i-d_1}b_i$. The bound comes from Lemma~\ref{lem:bound}, noting that $p\nmid d_1$, and the congruence modulo $d_1(p-1)$ guarantees that the elements are invariant.
\end{proof}

\begin{example}[$r=1$, $\vec{d} = \{ 2, 1 \}$, $p =3$, $g =3$] \label{ex:r=1distinct}
    Applying Theorem~\ref{T:normal}, the standard form of a curve with these parameters is 
    $$y^3-y=x^2+ax+\frac{b}{x}.$$
    Proposition~\ref{prop:r=1diff} implies that the invariants $I_1=ab,\, I_2=a^4,$ and $I_3=b^4$ generate the invariant ring.  This agrees with \cite[Corollary~4.6]{win6}.
\end{example}

When $d_1=d_2 =: d$, there is a second isomorphism that swaps the poles $P_\infty$ and $P_0$, given by 
$ (x,y) \mapsto \left (\frac{1}{\alpha x}, \lambda y \right) $
where $\alpha \in \Fpbar$, $\lambda \in \F_p^\times$ and $\alpha^d b_d = \lambda$. This second isomorphism acts on the coefficients as $(a_i,b_i)\mapsto\left(\frac{\alpha^i}{\lambda}b_i,\frac{1}{\lambda \alpha^i}a_i\right)$ where $a_d=1$. In this case, we reason analogous to Proposition~\ref{prop:r=1diff} and obtain the following. 

\begin{proposition}\label{prop:r=1equal}\
Suppose $d_1 = d_2 =:d$. With the notation of \eqref{eq:notation}, the ring of invariants of the curve \eqref{eq:r=1} is generated by
\begin{align*}
\Big \{ b_d^{(p-1)/2}, a_1^{k_1}\dots a_{d-1}^{k_{d-1}}b_1^{n_1}\dots b_{d}^{n_{d}} &+ a_1^{n_1}\dots a_{d-1}^{n_{d-1}}b_1^{k_1}\dots b_{d-1}^{k_{d-1}}b_d^{n_d+\frac{1}{d}\sum_{i=1}^{d-1}i(n_i-k_i)} : \\ n_d< (p-1)/2,\; & \sum_{i=1}^{d-1}i(n_i-k_i)\geq0,
\\\sum_{i=1}^{d-1} k_i + \sum_{i=1}^{d-1} n_i\leq d(p-1),\; &\sum_{i=1}^{d-1} k_i(i-d)-\sum_{i=1}^{d} n_i(i+d)\equiv0 \!\! \pmod{d(p-1)}
\Big \}.
\end{align*}
\end{proposition}

\begin{example}[$r=1,\, d=2,\, p=5,\,g=8$]  In this case, the standard form \eqref{eq:r=1} is 
     $$y^5-y = x^2 + ax + \frac{b_1}{x} + \frac{b_2}{x^2}.$$
 Proposition~\ref{prop:r=1equal} produces a list of 34 invariants that generate the invariant ring.  We use the \MAGMA~intrinsic \texttt{MinimalAlgebraGenerators()} to find the following minimal set of 34 generators.
 \begin{eqnarray*}
     b_2,\,
     ab_1,\,
     a^4b_2^2 + b_1^4,\,
     a^4b_2^3 + b_1^4,\,
     a^5b_1b_2^2 + ab_1^5,\,
     a^4b_2^4 + b_1^4,\,
     a^2b_1^2b_2^4 + a^2b_1^2,\,\\
     a^5b_1b_2^3 + ab_1^5,\,
     a^4b_2^5 + b_1^4,\,
     a^6b_1^2b_2^2 + a^2b_1^6,\,
     a^5b_1b_2^4 + ab_1^5,\,
     a^4b_2^6 + b_1^4,\,
     a^8b_2^4 + b_1^8
 \end{eqnarray*}
 \end{example}

Once again, we introduce a new standard form that is more amenable to invariant computation. Here, we simply send the two distinguished poles to $P_\infty$ and $P_0$ and make the polynomial part monic.

\begin{proposition} \label{prop:2pole}
    Consider an Artin-Schreier curve as given in \eqref{eqn:ASgeneralform} that has two poles of respective orders $d_1 \ge d_2$ such that all the other poles $\theta_3, \ldots , \theta_{r+1}$ (if any) are distinct from $P_\infty, P_0$ and have pole orders distinct from $d_1$ and $d_2$. Then $C_f$ is isomorphic to a curve of the form \begin{equation}\label{eq:11gen}y^p-y=x^{d_1}+P(x)+\frac{Q(x)}{x^{d_2}}+\sum_{i=3}^{r+1}\frac{R_i(x-\theta_i)}{(x-\theta_i)^{d_i}}\end{equation} 
    where $P(x)=\sum_{i=1}^{d_1-1}a_ix^i$,$Q(x)=\sum_{i=1}^{d_2}b_ix^{d_2-i}$ and $R_i(x)=\sum_{j=0}^{d_i-1}c_{ij}x^j,$
    with $a_i=b_i=0$ if $p\mid i$ for all $i$, and $c_{ij}=0$ if $p\mid j-d_i$, $c_{i,0}\neq0$ for all $i,j$.
\end{proposition}

Note that the polynomials and their coefficients appearing in Proposition~\ref{prop:2pole} are different from those defined in Theorem~\ref{T:normal} and \eqref{eq:notation}. However, if $\theta_3 = 1$, then they are related via $F(x) = a_{d_1}(x^{d_1} + P(x))$ and $Q(x) = a_{d_1}x^{d_2}G(1/x)$.

Just as for previous cases, the curve model given in \eqref{eq:11gen} is not unique. In addition to the isomorphisms described in \eqref{eq:ASisom}, if two poles $P_{\theta_i}, P_{\theta_j}$ have equal order, then swapping their indices does not change \eqref{eq:11gen}. This ambiguity is resolved by Proposition~\ref{prop:subgroup}. We first consider the action of $\mathcal{S}$, the product of permutations groups acting on the labeling of poles $\theta_i$ of equal order, before applying the isomorphisms described in \eqref{eq:ASisom}. 

The group $\mathcal{S}$ is a finite group that acts linearly on the coefficients of the polynomials $R_i$ and on the poles $\theta_i$. The invariant ring $\Fpbar[c_{ij},\theta_i]^\mathcal{S}$ can be computed for example with Corollary~\ref{multisymOrbits}. Write $\Fpbar[c_{ij},\theta_i]^\mathcal{S}=\Fpbar[e_i]$.  
 
\begin{theorem}\label{thm:11+} If $d_1\neq d_2$, then every isomorphism preserving the standard form given in \eqref{eq:11gen} is of the form $(x,y)\mapsto(\alpha x,\alpha^{d_1} y)$, with $\alpha^{d_1}\in\mathbb{F}_p^\times$, so $\alpha^{d_1(p-1)}=1$. 
These isomorphisms form a group that is isomorphic to $G=\mathbb{F}_p^\times$ acting linearly on $\Fpbar[a_i,b_i,e_i]$. The invariant ring $\Fpbar[a_i,b_i,c_{ij},\theta_i]^{G\times\mathcal{S}}=\Fpbar[a_i,b_i,e_i]^G$ can 
be computed using \cite[Algorithm 2.7.3]{sturmfelsIT}. 
\end{theorem}

\begin{remark} The action of $G$ is also defined on the subring $\Fpbar[a_i,b_i]$ and $K[a_i,b_i]^G$ is explicitly described in Proposition~\ref{prop:r=1diff}. 
\end{remark}

\begin{remark}
    If $d_1=d_2$, these ideas still apply, but the work becomes much more complicated, as illustrated in the following example.
\end{remark}

\begin{example}[{$r=3,\,\vec{d}=\{2,2,1,1\},\,p=3,\,g=8$}]\label{eg:2,2,1,1}
Following the discussion before Theorem~\ref{thm:11+}, we directly use the standard form
 $$y^3-y= x^2 + a_1x+\frac{b_1}{x} +\frac{b_2}{x^2}+\frac{e_1x+e_2}{x^2+e_3x+e_4}.$$
The acting group has order $4$ and is generated by  $\phi_{M,\lambda}$ with $(M,\lambda)=((\begin{smallmatrix}0 & 1\\ 1 & 0\end{smallmatrix}), 1)$ and $(\operatorname{Id},-1)$ with the notation of Remark~\ref{rem:auto}.
We apply Proposition~\ref{prop:subgroup} to the subgroup generated by $(M,\lambda)=(\operatorname{Id},-1)$. Since $p \ne 2$ (the non-modular case), Lemma~\ref{lem:bound} applies. Hence, the ring $\mathcal{R} \colonequals \overline{\F}_3[a_1,b_1,b_2,e_1,...,e_4]^{\langle((\operatorname{Id},-1))\rangle}$ can be presented as
\begin{equation}\label{eq:invLambda}
    \mathcal{R}=\overline{\F}_3[a_1^2,\,
    a_1b_1,\,
    a_1e_1,\,
    a_1e_3,\,
    b_1^2,\,
    b_1e_1,\,
    b_1e_3,\,
    b_2,\,
    e_1^2,\,
    e_1e_3,\,
    e_2,\,
    e_3^2,\,
    e_4].
\end{equation}
Applying the action of $(M,\lambda)=((\begin{smallmatrix}0 & 1\\ 1 & 0\end{smallmatrix}), 1)$ to the generators of $\mathcal{R}$ yields
\begin{align*}
&[a_1^2,\,
    a_1b_1,\,
    a_1e_1,\,
    a_1e_3,\,
    b_1^2,\,
    b_1e_1,\,
    b_1e_3,\,
    b_2,\,
    e_1^2,\,
    e_1e_3,\,
    e_2,\,
    e_3^2,\,
    e_4]\mapsto\\
&    [b_1^2/b_2,\,
    a_1b_1/b_2,\,
    b_1(e_1e_4-e_2e_3)/e_4^2,\,
    b_1e_3/e_4,\,
    a_1^2b_2,\,
    a_1(e_1e_4-e_2e_3)/e_4^2,\,
    a_1b_2e_3/e_4,\,
    \\
&   b_2,\,b_2(e_1e_4-e_2e_3)^2/e_4^4,\,
    b_2(e_1e_4-e_2e_3)e_3/e_4^3,\,
    -b_2e_2/e_4^2,\,
    b_2e_3^2/e_4^2,\,
    b_2/e_4].
\end{align*}
The group order is 2, so the invariants are generated by the multisymmetric functions on the generators of $\mathcal{R}$
up to degree 2. For the sake of brevity, we do not display them. 
\end{example}

\section{Three Poles}\label{sec:3poles}
Let $C_f$ be an Artin-Schreier curve defined over $\Fpbar$ where $f(x)$ has exactly three poles of order $d_1\geq d_2 \geq d_3$.  By Theorem~\ref{T:normal}, $C_f$ is isomorphic to a standard form curve $C_g:y^p-y=g(x)$, where \begin{equation}\label{eq:3pole}
    g(x)=F(x)+G\left(\frac{1}{x}\right)+H\left(\frac{1}{x-1}\right),
\end{equation} where $F$, $G$, and $H$ are polynomials with degree $d_1$, $d_2$, and $d_3$, respectively, and no monomial appearing in $F(x)$, $G(x)$, or $H(x)$ has an exponent that is divisible by $p$. This case is in some sense the simplest because the M\"obius transformation $M(x)$ in any isomorphism $\phi = \phi_{M,\lambda}$ as in \eqref{eq:ASisom} between standard forms is fully determined by the permutation of the three poles  $P_{\infty}$, $P_0$, and $P_1$, so 
$M(x)\in\langle\frac{1}{x}, 1-x\rangle$, corresponding to the subgroup generated by two transpositions  $(P_0 \ P_{\infty})$ and $(P_0 \ P_1)$ of the symmetric group of permutations of the poles. 

\begin{theorem}
    Any isomorphism preserving the standard form \eqref{eq:3pole}
    acts linearly on the set of coefficients of $F$, $G$, and $H$.
\end{theorem}
\begin{proof}
We only need to consider the M\"obius tranformations $M(x) = 1/x$ and $M(x) = 1-x$, and may assume that $\lambda = 1$ as scaling by $\lambda$ is a linear action on all the coefficients in the right hand side of \eqref{eq:3pole}. 

The isomorphism $\phi_{1/x,1}$ preserves the standard form \eqref{eq:3pole} if and only if $d_1=d_2$. The action of $\phi_{1/x,1}$ on \eqref{eq:3pole}
simply exchanges the coefficients of $F$ and $G$, which constitutes a linear action. It also sends the coefficients of $H(1/(x-1))$ to the coefficients of $H'(1/(x-1)) = H(-1-1/(x-1))$. The coefficients of $H'$ are easily seen to be linear in the coefficients of $H$.  

Similarly, $\phi_{1-x,1}$ preserves \eqref{eq:3pole} if and only if $d_2=d_3$. Its action swaps the coefficients of $G$ and $H$ and sends $F(x)$ to $F(1-x)$, again a linear action.
\end{proof}

The above shows that the isomorphisms of standard forms \eqref{eq:3pole} of Artin-Schreier curves with exactly three poles form a finite group; more specifically, a group isomorphic to a subgroup of $\mathfrak{S}_3\times \F_p^\times$ acting linearly on the set of coefficients of $F$, $G$, and $H$.  Thus, the infinite group of isomorphisms of Artin-Schreier curves acting on arbitrary models \eqref{eqn:ASgeneralform} reduces to a finite subgroup acting linearly on the coefficients of curves in standard forms, a finite dimensional vector space $V$ over $\Fpbar$.  Thus, invariants of these curves can in principle be computed relatively simply by algorithms implemented in $\MAGMA$ as described in Section~\ref{sec:MainAlg}, although the computations may take a prohibitively long time if there are poles are of high order.  

\subsection{Three poles of different orders} 
If all three poles have different orders, the computation of invariants is trivial.  Since all poles must be fixed by any isomorphism of standard forms, the only action is by $\lambda\in\F_p^{\times}$, and each coefficient of $f(x)$ is simply scaled by $\frac{1}{\lambda}$.  Thus, two such curves are isomorphic exactly when the vectors of free coefficients of their standard forms are $\F_p^{\times}$-multiples of each other. 
Using the notation of \eqref{eq:notation}, we obtain the following generators for the invariant ring. 

\begin{proposition}
   If $d_1>d_2>d_3$, then the invariants of the curve \eqref{eq:3pole} are generated by elements of the form $a_{d_1}^{p-2}{a_i}$, $a_{d_1}^{p-2}{b_j}$, and $a_{d_1}^{p-2}{c_k}$, where $ 1\leq i\leq d_1, 1\leq j\leq d_2, 1\leq k\leq d_3$, with  $i,j,k\not\equiv0\pmod{p}.$
\end{proposition}
    \begin{proof}
        Every standard form preserving isomorphism scales the coefficients of a curve in standard form \eqref{eq:3pole} by some $\lambda\in\F_p^{\times}$, so in particular $a_{d_1}^{p-1}$ is invariant.  It is easy to check that the $(p-1)$-st power of any coefficient is also invariant and, since the scalar $\lambda$ is in common for all coefficients, so will the other elements of the generating set.  These expressions are all algebraically independent, and their number is equal to the number of free coefficients in the standard form, whichin turn is equal to the dimension of the moduli space and thus the number of invariants.  Therefore, the set described above is a complete generating set for the invariant ring.
    \end{proof}
    
\begin{example}[$r=2$, $p=3$, $\vec{d}=\{4,2,1\}$, $g=8$]\label{ex:E=5,3,2}
 
Equation \eqref{eq:3pole} becomes 
 \[y^3-y=a_4x^4 + a_2x^2+a_1x+\frac{b_1}{x} + \frac{b_2}{x^2}+\frac{c}{x-1},\] where $a_4, b_2,c\neq0$.  Then the invariant ring is generated by $I_1= a_4a_1,$ $I_2=a_4a_2$, $I_3=a_4^2$, $I_4=a_4b_2$, $I_5=a_4b_1$, and $I_6=a_4c$.
 It is straightforward to see that this is a reconstructing set, and since the dimension of the moduli space component is $6$, these invariants must be algebraically independent.
\end{example}

\subsection{Two poles with same order, third of distinct order}
If two poles have the same order and the third has a different order, then the computation of invariants is also fairly straightforward. The poles of equal order can be exchanged, but the third pole must be fixed.
The isomorphisms preserving \eqref{eq:3pole} form a group $G$ isomorphic to $\mathfrak{S}_2 \times \F_p^\times$, which acts linearly on the coefficients of $g(x)$. If $p>2$, then $p$ does not divide $|G| = 2(p-1)$. So the generators of the invariant ring have degrees bounded by $2(p-1)$ by Lemma~\ref{lem:bound}.

\subsection{Three poles with same order}\label{subsec:3sameOrder}
This case is the most difficult case as demonstrated by even the simplest example discussed in \cite[Sec.\ 4.4.3]{win6}.  In this case, the group $G \cong \mathfrak{s}_3 \times \F_p^\times$ of isomorphisms of standard forms has order $6(p-1)$ and acts linearly on the coefficients of $g(x)$. When $p>3$ (the non-modular case), the degree of each generator in the coefficients is bounded by $|G|$ by Lemma~\ref{lem:bound}. However, when $p=3$ (the modular case), we need to resort to Lemma~\ref{lem:boundp} which gives a much larger bound on these degrees. We implemented the required algorithms in \MAGMA.  The next smallest example beyond the one given in \cite[Sec.\ 4.4.3]{win6} has parameters $r=2$, $p=3$, $\vec{d}=\{2,2,2\}$, $g=8$.
In this case, the \MAGMA~computation of the invariant ring using the linear action did not terminate due to insufficient memory.

\subsection{Generalizations from the three pole case}\label{ssec:3plus}

Consider an Artin-Schreier curve $C_f$ that has more than three poles but with three pole orders different from all hte other pole orders; for example, $\vec{d}=\{1, 1, 1, 2\}$ or $\vec{d}=\{1,1,2,4,5\}$. In this case, we can take advantage of our work in the three pole case to ease the computation of invariants.  First, we explicitly describe the standard model employed for this case (as described in Section~\ref{sec:MainAlg}).
\begin{theorem}\label{thm:standard3plus}
    Let $C_f$ be an Artin-Schreier curve defined over $\Fpbar$ where $f(x)$ has $r+1$ poles with three poles of order disjoint from the orders of the remaining poles. Choose the subset of this type with largest pole orders to be $d_1,$ $ d_2$, and $d_3$, with $d_1\geq d_2 \geq d_3$.  List the remaining pole orders in decreasing sequence and denote them $d_4, \dots, d_{r+1}$.  Then $C_f$ is isomorphic to a curve $C_g:y^p-y=g(x)$, where \begin{equation}\label{eq:3+form}
    g(x)=F(x)+G\left(\frac{1}{x}\right)+H\left(\frac{1}{x-1}\right)+\sum_{j=4}^{r+1}J_j\left(\frac{1}{x-\theta_j}\right),
\end{equation} with $F(x)=\sum_{i=1}^{d_1}a_ix^i$, $G(x)=\sum_{i=1}^{d_2}b_ix^i$,  $H(x)= \sum_{i=1}^{d_3}c_ix^i$, $J_j(x)=\sum_{i=1}^{d_j}e_{ij}x^i,$ where $a_i=b_i=c_i=e_{ij}=0$ for any $i\equiv 0\pmod{p}$ and $a_{d_1},b_{d_2},c_{d_3},e_{jd_j}\neq 0$ .  
\end{theorem}

As before, the curve model in \eqref{eq:3+form} is not unique. We resolve this by using Proposition~\ref{prop:subgroup} and first considering the action of $\mathcal{S}$, the permutations group acting on the labeling of poles $\theta_i$ of equal order, before applying the isomorphisms described in \eqref{eq:ASisom}. The invariant ring $\Fpbar[e_{ij},\theta_i]^\mathcal{S}$ is computable with Corollary~\ref{multisymOrbits}. Write $\Fpbar[e_{ij},\theta_j]^\mathcal{S}=\Fpbar[f_k]$, where the $f_k$ are multisymmetric functions in $\{e_{ij},\theta_j\}$.

Now we consider the true isomorphisms as in \eqref{eq:ASisom}. The M\"{o}bius transformations~$M$ acting on this standard form must preserve the set $\{P_{\infty}, P_0,P_1\}$ and hence form a group $G$ isomorphic to a subgroup of $\mathfrak{S}_3$.  We showed that these act $\Fpbar$-linearly on the coefficients of $F$, $G$, and $H$. Furthermore, the $\lambda$-action uniformly scales all coefficients of $F$, $G$, $H$, $J_4, \cdots , J_{r+1}$ by $\lambda^{-1}$. This resolves the model  ambiguity, so can now compute invariants of $G$ acting on $a_i, b_i, c_i, f_i$.  In practice, it is difficult to write down the general action precisely, but we can handle some special cases.

\begin{theorem}\label{thm:3unique}  If $d_1>d_2>d_3$ and $\{d_1,d_2,d_3\} \cap\{d_4,\dots d_{r+1}\} = \emptyset$, then every isomorphism preserving the standard form in Theorem~\ref{thm:standard3plus} is $\phi_{Id,\lambda}$ for $\lambda\in\F_p^{\times}$. These isomorphisms form a group that acts linearly on $K[a_i,b_i,c_i,g_i,f_i]$. The invariant ring $K[a_i,b_i,c_i,g_i,f_i]^G$ can 
be computed using \cite[Algorithm~2.7.3]{sturmfelsIT}. 
\end{theorem}

\begin{example} A simple example 
is in $r=4$, $p=3$, $\vec{d}=\{1,1,2,4,5\}$, $g=16$.  This corresponds to a dimension $13$ irreducible component of the moduli space $\mathcal{AS}_{16,8}$.  We choose a standard model of the form $y^3-y=f(x)$, where $f(x)$ equals
\[a_5x^5+a_4x^4+a_2x^2+a_1x+ 
\frac{b_1}{x} + \frac{b_2}{x^2} + \frac{b_3}{x^3} + \frac{b_4}{x^4}
+\frac{c_1}{x-1} + \frac{c_2}{(x-1)^2}
+ \frac{e_1}{x-{\theta_1}}+\frac{e_2}{x-\theta_2},\] where $\theta_1\neq \theta_2$, $\theta_i\not\in\{0,1\}$ for $i=1,2$, and $a_5, b_4, c_2, e_1, e_2\neq 0$. 

From the discussion above, we see that $P_{\infty}$, $P_0$ and $P_1$ must be fixed by any isomorphism of standard forms, as in the three poles of multiplicity one case.  Thus, the functions $I_1=a_5a_1, I_2=a_5a_2, I_3=a_5a_4, I_4=a_5^2, I_5=a_5b_4, I_6=a_5b_2, I_7=a_5b_1, I_8=a_5c_2, \textrm{ and }I_9=a_5c_1,$
are invariant, as well as $I_{10}=a_5(e_1+e_2), I_{11}=\theta_1+\theta_2, I_{12}=a_5(e_1\theta_1+e_2\theta_2), I_{13}=a_5^2e_1e_2, \textrm{ and }I_{14}=\theta_1\theta_2,$
the additional invariants arising from the action of the symmetric group $\mathfrak{S}_2$ on the poles $P_{\theta_1}$ and $P_{\theta_2}$, which have no natural ordering in the standard form.  Note that the set of invariants $\{I_1, I_2,\dots, I_{14}\}$ forms a reconstructing system. 
We also have the  relation $I_{12}(I_{10}I_{11}-I_{12})=I_{14}(I_{10}^2+I_{13})+I_{13}(I_{11}^2+I_{14}).$
\end{example}

When all poles have multiplicity one, we can explicitly describe all the invariants.

\begin{corollary}\label{cor:allunique}  If $C_f$ is an Artin-Schreier curve in form \eqref{eq:3+form} with at least $4$ poles, and all poles have multiplicity one, then the invariants ring of the curve is generated by the invariants listed in Theorem~\ref{thm:3unique} and those in the set $\{a_{d_1}^{p-2}e_{ij},\theta_j: 1\leq i\leq d_j, p\nmid i, 4\leq j\leq r+1\}.$
\end{corollary}

\begin{example}[$r=3$, $p=3$, $\vec{d}=\{5,4,2,1\}$, $g=14$]
The standard form of this curve is $y^3-y=f(x)$ with \[f(x)= a_5x^5+a_4x^4+a_2x^2+a_1x+ 
\frac{b_1}{x} + \frac{b_2}{x^2} + \frac{b_4}{x^4}
+ \frac{c_1}{x-1} + \frac{c_2}{(x-1)^2}
+ \frac{e}{x-{\theta}},\] where $\theta\not\in\{0,1\}$ and $a_5, b_4, c_2, e\neq 0$. 
Again,  the functions $I_1=a_5a_1, I_2=a_5a_2, I_3=a_5a_4, I_4=a_5^2, I_5=a_5b_4, I_6=a_5b_2, I_7=a_5b_1, I_8=a_5c_2, \textrm{ and }I_9=a_5c_1$ are invariant. Since the pole $P_{\theta}$ is also of multiplicity one, we obtain the further invariants $I_{10}=a_5e$ and $I_{11}=\theta$.  The set $\{I_1, I_2,\dots, I_{11}\}$ is a reconstructing system, and its cardinality 11 is equal to the dimension of the corresponding component of the moduli space $\mathcal{AS}_{14,6}$ by Theorem~\ref{theorem:PriesZhuDimension}. 
\end{example}

The curve model given in \eqref{eq:3+form} is also well suited to the case of exactly four poles with one pole order different from all the others.  In this case, there is no ambiguity in the ordering of non-distinguished poles, so $\mathcal{S}$ is trivial. The following example illustrates this situation.

\begin{example}\label{ex:2,2,2,3}($r=3$, $p=3$, $\vec{d}=\{1,1,1,2\}$, $g=8$).  This curve has standard model
\[y^3-y=ax+\frac{b}{x}+\frac{c}{(x-1)}+\frac{e_1}{(x-\theta)}+\frac{e_2}{(x-\theta)^2},\] where $a, b, c, e_2,\theta\neq 0$ and $\theta\neq 1$.  The group acting is $\mathfrak{S}_3\times\F_3^{\times}$, with the M\"{o}bius transformations acting as $\mathfrak{S}_3$ given by $M(x)\in\{x,\frac{1}{x},1-x,\frac{1}{1-x}, \frac{x-1}{x}, \frac{x}{1-x}\}$.  Using the notation from above, we observe that the  coefficient vector $V=(a,b,c,e_1,e_2,\theta)$ is sent to one of the following:
\begin{align*}
&    (a,b,c,e_1,e_2,\theta),\, (b,a,-c,\tfrac{-e_1+e_2}{\theta}, \tfrac{-e_2}{\theta}, \tfrac{1}{\theta}),\, (-a,-c,-b,-e_1,e_2,\theta-1),\,\\ 
&    (-c,-a,-b,\tfrac{e_1+e_2}{\theta-1}, \tfrac{-e_2}{\theta-1}, \tfrac{1}{\theta-1}),\, (-b,c,-a,\tfrac{e_1-e_2}{\theta}, \tfrac{-e_2}{\theta}, \tfrac{1-\theta}{\theta}),\, (a,c,b,e_1,-e_2,\theta-1),\\
&   (c,-b,a,\tfrac{e_1+e_2}{\theta-1}, \tfrac{-e_2}{\theta-1}, \tfrac{\theta}{\theta-1}),\,(-a,-b,-c,-e_1,-e_2,\theta),\, (-b,-a,c,\tfrac{e_1-e_2}{\theta}, \tfrac{e_2}{\theta}, \tfrac{1}{\theta}),\\ 
&   (c,a,b,\tfrac{-e_1-e_2}{\theta-1}, \tfrac{e_2}{\theta-1}, \tfrac{1}{\theta-1}),\, (b,-c,a,\tfrac{-e_1+e_2}{\theta}, \tfrac{e_2}{\theta}, \tfrac{1-\theta}{\theta}),\, (-c,b,-a,\tfrac{-e_1-e_2}{\theta-1}, \tfrac{e_2}{\theta-1}, \tfrac{\theta}{\theta-1}).
\end{align*}

For $a$, $b$, and $c$, we obtain the same invariants as in \cite[Corollary 4.16]{win6}, namely $I_1=(abc)^2$, $I_2=(abc)(a-b-c)$, $I_3=ab+ac-bc$, and $I_4=a^2+b^2+c^2$ with the relation $I_1(I_3+I_4)=I_2^2$. We further have the invariant $I_5= \frac{(\theta^2-\theta+1)^3}{\theta^2(\theta-1)^2}$, obtained as the sum of the orbit of $\theta$.  We can check that $I_6=e_2^2\frac{\theta^2(\theta-1)^2+\theta^2+(\theta-1)^2}{\theta^2(\theta-1)^2}$ and $I_7=e_1^2+\frac{(e_1-e_2)^2}{\theta^2}+\frac{(e_1+e_2)^2}{(\theta-1)^2}$ are invariant.  The set $\{I_1,I_2,I_3,I_4\}$ allows for the reconstruction of $a$, $b$ and~$c$, while the invariant $I_5$ allows reconstruction of $\theta$ (can solve for six possible values, all valid for some form).  Having chosen a valid $\theta$, we can use $I_6$ and $I_7$ to determine a valid pair $(e_1,e_2)$.
\end{example}

\section{Four Poles}\label{sec:fourPoles}

The scenario of 4 poles of equal order is difficult and not covered by the generalizations in previous sections. We present the special case where the 4 poles have order 1 in detail, as it already illustrates the added complexities that need to be considered in the general case of 4 poles of equal order.

\subsection{Four poles of equal order}
The standard form in this case is
$$
y^p-y=ax+\frac{b}{x}+\frac{c}{x-1}+\frac{e}{x-\theta}.
$$

Any isomorphism between curves of this form will send $\theta$ to one element in the set $\{\theta, 1-\theta, \frac{1}{\theta}, \frac{1}{1-\theta}, \frac{\theta-1}{\theta}, \frac{\theta}{\theta-1}\}$, so $j\colonequals \frac{(\theta^2-\theta+1)^3}{\theta^2(\theta-1)^2}$ will remain invariant. Denote by $P_j(x)$ the polynomial $(x^2-x+1)^3-jx^2(x-1)^2$, and put $k = \Fpbar(j)$.

Let $G$ be the finite subgroup of order $24$ of $\operatorname{PGL}_2(k)$ sending $\{P_\infty, P_0,P_1,P_\theta\}$ to $\{P_\infty,P_0,P_1,P_{\theta'}\}$ in some order.
We consider the action (change of variables plus eliminating the constant term) of this group on the ring $R=k[a,b,c,e,\theta]/(P_j(\theta))$ with the graduation $R=\oplus_{i\geq0}R_i$ given by $\deg(a)=\deg(b)=\deg(c)=\deg(e)=1$ and $\deg(\theta)=0$. The $k$-vector spaces $R_i$ are finite dimensional and the action of $G$ on them is linear.  The group $G$ is generated by
\begin{align*}
&M_1\colon 
[a,b,c,e,\theta]\mapsto[{b}, a, -{c}, -\tfrac{e}{\theta^2}, \tfrac{1}{\theta}]
\\
&M_2\colon
[a,b,c,e,\theta]\mapsto[c,-b,a,\tfrac{-e}{(\theta-1)^2},\tfrac{\theta}{\theta-1}]
\\
&M_3\colon
[a,b,c,e,\theta]\mapsto[\tfrac{-e}{\theta(\theta-1)}, \tfrac{b(\theta-1)}{\theta}, \tfrac{c\theta}{\theta-1}, a\theta(1-\theta), 1-\theta]
\end{align*}

We can use classical techniques to compute a generating system of invariants.

\begin{theorem} The invariant ring $R^G$ is the ring of integers of the fraction field $k(J_1, J'_2, J_2, J_3,J_4)$, for which $J_1=(\theta^2-\theta+1)\left(a^2+\tfrac{b^2}{\theta^2}+\tfrac{c^2}{(1-\theta)^2}+\tfrac{e^2}{\theta^2(1-\theta)^2}\right),\,$ $J'_2=ab-bc+ca+e\left(a-\tfrac{b}{\theta^2}-\tfrac{c}{(\theta-1)^2}\right)$, $J_2=a^2b^2+b^2c^2+c^2a^2+e^2\left(a^2+\tfrac{b^2}{\theta^4}+\tfrac{c^2}{(\theta-1)^4}\right),$
$J_3=a^2b^2c^2+\tfrac{a^2b^2e^2}{\theta^2}+\tfrac{a^2c^2e^2}{(\theta-1)^2}+\tfrac{b^2c^2e^2}{\theta^2(1-\theta)^2},\,J_4=\tfrac{abce}{1-\theta+\theta^2}.$

\end{theorem}
\begin{proof} We only need to compute the action of a finite group acting linearly on $R$. There are  algorithms to effect this, several of them implemented in \MAGMA, but they do not extend to the case where the base field is a function field such as $k = \Fpbar(j)$ as in our current situation. Our ad-hoc implementation takes very long, especially for checking the algebraic independence of the obtained generators. For $p = 3$, since $3\mid 24$, we need to check for generators up to degree $5\times (24-1)$. Instead, we apply a Remark~\ref{genericallyreconstruction} for generically reconstructing systems, which is a stronger version of Theorem~\ref{thm:reconstruct}. So we need to prove that we can generically reconstruct from the invariants (it is straightforward to check that they are invariants) in the statement of the theorem. 

We start by fixing a value of $\theta$ as a root of  $P_j(x)=0$. Next, we compute $\{a^2,\frac{b^2}{\theta^2},\frac{c^2}{(1-\theta)^2},\frac{e^2}{\theta^2(1-\theta)^2}\}$ from their symmetric functions: $J_1,\,J_2,\,J_3,\,J_4^2$. From $a^2$ in that set, we derive $a$. Now $J_4$ produces $6\times4\times1$ possibilities for $b,\,c,\,e$ from  the values $\{\frac{b^2}{\theta^2},\frac{c^2}{(1-\theta)^2},\frac{e^2}{\theta^2(1-\theta)^2}\}$. We evaluate $J'_2$ at these 24 possibilities and obtain different polynomials. Hence, we can generically obtain the correct values of $\{b,c,e\}$ and we can reconstruct. 
\end{proof}
The $\lambda$-action in \eqref{eq:ASisom} can be easily computed as explained after Proposition~\ref{prop:subgroup} and we omit the details here. The same techniques developed in this section apply to the setting of 4 poles of equal order exceeding 1.

\section{Five Poles}\label{sec:5poles}

An even more difficult case is given by 5 or more poles of equal order. As in the previous section, we only provide details for the case where this pole order is $1$; the general case proceeds in a similar fashion. Already in this setting, we encounter the major challenge that arises for Artin-Schreier curves with many poles of equal order, namely that the finite set of standard form preserving isomorphisms no longer forms a group. Here, we apply then Remark~\ref{GnotGroup} and Theorem~\ref{thm:GnotGroup}. 

Consider an Artin-Schreier curve with 5 poles of order $1$. In general, such a curve is given by equations of the form $y^p-y=\frac{P_5(x,z)}{Q_5(x,z)}$ with $\deg(P_5) \le \deg(Q_5) = 5$, $P_5$ and~$Q_5$ coprime, $Q_5$ square-free, and $x \nmid P_5(x)$ (for eliminating $p$-power monomials). The group $\tilde{G}=\operatorname{PGL}_2(\Fpbar)\times\mathbb{F}_p^\times$  acts on their coefficients as described in \eqref{eq:ASisom}. So~$\tilde{G}$ acts on an open Zarisky set $U$ of $\mathbb{P}^{10}$. The computation of the invariants for the action of this infinite group is out of reach at present. Again, our idea is to consider standard models with fewer parameters that are only preserved by a finite number of elements of $\tilde{G}$. Write the curve as
\begin{equation}\label{eq:5poles1}
y^p-y=
ax+\frac{b}{x}+\frac{c}{x-1}+\frac{tx+r}{x^2-sx+u}.
\end{equation}
where $x^2-sx+u=(x-\theta_1)(x-\theta_2)$. We do not want to distinguish $\theta_1$ and $\theta_2$ to avoid imposing an order on the poles. On the other hand, we need to consider them as different in order to well-define the action of the finite set of standard form preserving isomorphisms, which has cardinality $5\times4\times3$. Each such isomorphism is determined by the pre-images of $P_\infty$, $P_0$ and $P_1$. For example, we have:
\begin{eqnarray*}
&&M_1\colon[\infty, 0,1,\theta_1,\theta_2]\mapsto \left [0,\infty,1,\frac{1}{\theta_1},\frac{1}{\theta_2} \right ]
\\[2pt]
&&M_2\colon [\infty, 0,1,\theta_1,\theta_2]\mapsto  \left[1,0,\infty, \frac{\theta_1}{\theta_1-1},\frac{\theta_2}{\theta_2-1} \right ]
\\[2pt]
&&M_{3,1}\colon [\infty, 0,1,\theta_1,\theta_2]\mapsto \left [1-\theta_1,0, 1,\infty, \frac{\theta_2(1-\theta_1)}{\theta_2-\theta_1} \right ]
\\[2pt]
&&M_{3,2}\colon [\infty, 0,1,\theta_1,\theta_2]\mapsto \left [1-\theta_2,0, 1,\frac{\theta_1(1-\theta_2)}{\theta_1-\theta_2},\infty \right ].
\end{eqnarray*}
These transformations do not define a group since we cannot distinguish between the images of $\theta_1$ and $\theta_2$ (i.e.\ the ``new'' values of $\theta_1$ and $\theta_2$) to compose these maps. Writing $\theta = \theta_1$ and substituting $u=\theta(s-\theta)$,
we obtain for example the following:
$$
M_{3,1}\colon\,[a,b,c,s,u,t,r]\mapsto \left [\frac{t\theta+r}{\theta(\theta-1)(s-2\theta)}, \, \frac{b(\theta-1)}{\theta}, \, \frac{c\theta}{\theta-1}, \, \frac{(1-\theta)(2s-3\theta)}{(s-2\theta)}, \right . $$\, 
$$\frac{(1-\theta)^2(s-\theta)}{(s-2\theta)}, \, \frac{\theta(\theta-1)(as^3 - 6as^2\theta + 12as\theta^2 - 8a\theta^3 - st + t\theta - r)}{(s-2\theta)^3},$$ 
$$ \left . \frac{\theta(\theta-1)(as^3 - 5as^2\theta + 8as\theta^2 - 4a\theta^3 - st + t\theta - r)}{(s-2\theta)^3} \right ]
$$
In this way we define the orbits of the coefficients $[a,b,c,s,u,t,r]$ as in Remark~\ref{GnotGroup} and we compute invariants via multisymmetric functions as in Theorem~\ref{thm:GnotGroup}. 

\begin{theorem} Any set of generators of $K[U]^{\tilde{G}}$ evaluated at the the coefficients in~\eqref{eq:5poles1} produces generators of the multisymmetric functions as given in Theorem~\ref{thm:multisym} for the elements of the orbit of $[a,b,c,s,u,t,r]$. 
\end{theorem}

\section{Conclusion}\label{sec:conclusion}

The theory of invariants for Artin-Schreier curves is surprisingly rich, presenting both conceptual and computational challenges even for relatively small genus.  Computationally, the main obstacles are working with algebraic extensions of polynomials rings or such rings over function field. Perhaps even more importantly, the number of generators of the invariant rings grows very rapidly with the genus, making it infeasible to prove that the sets of invariants that we compute are minimal; this requires Gr\"oebner basis computations in as many variables as the number of invariants.
Theoretically, there is a need to build a stronger theory of specializing invariants of infinite groups to models where the action can no longer described as a group, as we saw in several cases.  With these advances, invariant computation would be within reach for many more curves.

\bibliographystyle{alpha}
\bibliography{synthbib.bib}

\end{document}